\documentclass[11pt]{amsart}

\usepackage{amsmath,amsthm,amssymb,amscd}
\usepackage[margin=1in]{geometry}
\usepackage[bookmarks]{hyperref}
\usepackage{verbatim}
\usepackage{xcolor}
\usepackage{rank-2-roots}

\newtheorem{theorem}{Theorem}[subsection]
\newtheorem{proposition}[theorem]{Proposition}

\newtheorem{corollary}[theorem]{Corollary}
\newtheorem{lemma}[theorem]{Lemma}

\theoremstyle{definition}

\newtheorem{definition}[theorem]{Definition}
\newtheorem{example}[theorem]{Example}

\newtheorem{problem}[theorem]{Problem}

\numberwithin{equation}{subsection}
\numberwithin{table}{subsection}

\newcommand{\Ext}{\operatorname{Ext}}

\newcommand{\0}{\bar 0}

\newcommand{\1}{\bar 1}

\newcommand{\ch}{\operatorname{ch}}

\begin{document}

\title[On Sheaf Cohomology  for Supergroups]
{\bf On Sheaf Cohomology for Supergroups Arising from Simple Classical Lie Superalgebras}

\author{\sc David M. Galban}
\address
{Department of Mathematics\\ University of Georgia \\
Athens\\ GA~30602, USA}
\thanks{Research of the first author was supported in part by NSF
RTG grant 1344994.}  

\email{david.galban25@uga.edu}

\author{\sc Daniel K. Nakano}
\address
{Department of Mathematics\\ University of Georgia \\
Athens\\ GA~30602, USA}
\thanks{Research of the second author was supported in part by NSF
grants  DMS-1701768 and DMS-2101941.}

\email{nakano@uga.edu}

\dedicatory{In memory of James E. Humphreys} 


\date\today

\maketitle

\begin{abstract} In this paper the authors study the behavior of the sheaf cohomology functors $R^{\bullet}\text{ind}_{B}^{G}(-)$ where $G$ is 
an algebraic group scheme corresponding to a simple classical Lie superalgebra and $B$ is a BBW parabolic subgroup as defined in \cite{GGNW}.
We provide a systematic treatment that allows us to study the behavior of these cohomology groups $R^{\bullet}\text{ind}_{B}^{G} L_{\mathfrak f}(\lambda)$ 
where $L_{\mathfrak f}(\lambda)$ is an irreducible representation for the detecting subalgebra ${\mathfrak f}$. In particular, we prove an analog of 
Kempf's vanishing theorem and the Bott-Borel-Weil theorem for large weights. 
\end{abstract} 

\vskip 1cm 

\section{Introduction}

\subsection{} Let $G$ be a reductive algebraic group over an algebraically closed field $k$. If $B$ is a Borel subgroup of $G$ then the sheaf cohomology 
groups: 
$$H^{\bullet}(\lambda):={\mathcal H}^{\bullet}(G/B,{\mathcal L}(\lambda))\cong R^{\bullet}\text{ind}_{B}^{G}\ \lambda$$ 
play a central role in the representation theory for $G$. It is well-known that the irreducible $G$-modules are indexed by dominant integral weights, $X_{+}$, 
and can be realized as the socles of $H^{0}(\lambda)$. More precisely, for any $\lambda\in X_{+}$, one has $L(\lambda)=\text{soc}_{G}H^{0}(\lambda)$. 
Another result that holds over arbitrary $k$ is Kempf's vanishing theorem which states that for $\lambda\in X_{+}$, $H^{n}(\lambda)=0$ for $n>0$. When 
the field $k$ is of characteristic zero, the rational representations of $G$ are completely reducible, and a description of $H^{\bullet}(\lambda)$ is given 
via the classical Bott-Borel-Weil (BBW) theorem. For fields of characteristic $p>0$, the general vanishing behavior for $H^{\bullet}(\lambda)$ is not known, and it is not clear 
how to formulate an appropriate generalization of the BBW theorem. Additional information in regards to vanishing behavior  
is due to Andersen for $n=1$ \cite{And} where he described $\text{soc}_{G}H^{1}(\lambda)$ for all weights $\lambda$. 

Now consider the more general case when $G$ is a supergroup scheme with $\text{Lie }G={\mathfrak g}$ where ${\mathfrak g}$ is a classical ``simple" Lie superalgebra over ${\mathbb C}$, and $P$ a parabolic 
supergroup scheme of $G$. A central problem in the super-representation theory is to understand the behavior of the sheaf cohomology groups 
$R^{i}\text{ind}_{P}^{G}(-)$. Zubkov has published a general discussion on this topic \cite{Zubkov}. Specific calculations of sheaf cohomology have been made 
for specific supergroups such as ${GL}(m|n)$, ${OSP}(m|2n)$, and ${Q}(n)$ with certain parabolic/Borel subgroups \cite{Zubkov,GrS1,GrS2,P,PS,S1,S2}. 
From these computations, it is not clear if there is a general theory that can be applied for all classical simple Lie superalgebras like the one for reductive algebraic groups 
where computations of sheaf cohomology is related to the combinatorics of finite reflection groups. The existence of a Bott-Borel-Weil theory for supergroups has been an open questions since the 1980s. Recently, Coulembier has developed a BBW theory for basic classical Lie superalgebras \cite{Col}. 

D. Grantcharov, N. Grantcharov, Nakano and Wu \cite{GGNW} introduced a family of parabolic subgroups for $G$ called {\em BBW parabolic} subgroups. 
These parabolic subgroups arise naturally when considering the detecting subalgebras as defined by Boe, Kujawa and Nakano in the mid 2000s \cite{BKN1}. In 
\cite{GGNW}, it was shown that if $B$ is a BBW parabolic, the polynomial $p_{G,B}(t)=\sum_{i=0}^{\infty} \dim R^{i}\text{ind}_{B}^{G} {\mathbb C}\ t^{i}$ is equal to 
a Poincar\'e polynomial for a finite reflection group $W_{\1}$ specialized at a power of $t$. The existence of the BBW parabolics was also used in 
\cite{GGNW} to resolve a 15 year old conjecture posed in \cite{BKN1} on the realization of the cohomological support varieties for the pair $({\mathfrak g},{\mathfrak g}_{\0})$ as 
a rank variety over a detecting subalgebra. 

\subsection{} The main goal of this paper is to provide a systematic treatment of the higher sheaf cohomology groups $R^{\bullet}\text{ind}_{B}^{G}L_{\mathfrak f}(\lambda)$ 
where $B$ is BBW parabolic subgroups and $L_{\mathfrak f}(\lambda)$ is an irreducible representation for the detecting subalgebra ${\mathfrak f}$. 
For the reader who is familiar with the setting for reductive algebraic groups, we provide a treatment similar to the one presented in \cite[Part II]{Jan}. Many of the constructions are analogous. However, for supergroups, there are differences in the behavior of $H^{n}(\lambda)$ that will be indicated at various points in the paper. 

The paper is organized as follows. In Section 2, we present the basic notion and conventions that will be use throughout the paper. In particular, the BBW parabolic subgroups and subalgebras from \cite{GGNW} are reintroduced. In the following section (Section 3), we construct an important spectral sequence that allows us to compare the sheaf cohomology for a supergroup scheme 
$G$ when restricted to the even component $G_{\0}$ to the sheaf cohomology for $G_{\0}$. Applications are presented that demonstrate how several of the key results from \cite{GGNW} can 
be obtained through this new approach. 

In Section 4, it is shown that, as in the reductive group situation, the socles of the induced modules $H^{0}(\lambda)$ where $G$ is a supergroup arising from a simple classical 
Lie superalgebra and $B$ is a BBW parabolic subgroup are either zero or a simple $G$-module. This yields a classification of irreducible rational $G$-representations. The vanishing behavior of the higher sheaf cohomology modules $H^{n}(\lambda)$ is explored in Section 5. 
In particular, we prove an analog of Kempf's vanishing theorem and the Bott-Borel-Weil Theorem in this context. Examples are provided for $Q(n)$, $n=2,3$. 

Finally, in Section 6, our results are applied to obtain results on $H^{1}(\lambda)$. In the reductive group case the socles are either zero or irreducible for these modules. 
In \cite{GGNW}, it was shown this need not be the case for supergroups $G$. We provide a formula for calculating the socle of $H^{1}(\lambda)$ for $G$, and give a criterion
for the irreducibility of the socle of $H^{1}(\lambda)$. These ideas were inspired by the earlier work of Andersen on $H^{1}(\lambda)$ for reductive groups. 
Our analysis is then applied to show that for $G=Q(n)$, $H^{1}(\lambda)$ has simple socle for all non-dominant weights $\lambda$. 

\subsection{Acknowledgements} The authors acknowledge the many contributions James E. Humphreys (``Jim") made in the area of algebraic group representations. 
His books and expository articles were instrumental in the development of the field, and he also wrote several important papers on the structure of line bundle cohomology 
\cite{Hum1, Hum2, Hum3, Hum4, Hum5}. 

Jim was very supportive of younger mathematicians in representation theory. When the second author first entered the field, Jim wrote an extensive Mathematical Review (MR) on his 
Ph.D. thesis on {\em Projective Modules over Lie algebras of Cartan type} in 1990. The second author thanks Jim for his words of encouragement, careful reading of his work and for his guidance over the past 30 years.

\section{Overview: Key Concepts}\label{S:RepTheory}

\subsection{Notation: } We will use and summarize the conventions developed in
\cite{BKN1, BKN2, BKN3, LNZ, GGNW}. For more details we refer the reader to \cite[Section 2]{BKN1}. 

Let ${\mathfrak g}={\mathfrak g}_{\0}\oplus {\mathfrak g}_{\1}$ be a Lie superalgebra over the complex numbers ${\mathbb C}$ with 
supercommutator $[\;,\;]:{\mathfrak g}\otimes {\mathfrak g} \to {\mathfrak g}$. At times, we will impose more conditions on ${\mathfrak g}$. 
For instance, ${\mathfrak g}$ is called \emph{classical} if (i) there is a connected reductive algebraic group $G_{\0}$ such
that $\operatorname{Lie}(G_{\0})={\mathfrak g}_{\0},$ and (ii) the action of $G_{\0}$ on ${\mathfrak g}_{\1}$ differentiates to
the adjoint action of ${\mathfrak g}_{\0}$ on ${\mathfrak g}_{\1}.$  Furthermore, a Lie superalgebra ${\mathfrak g}$ is  \emph{basic classical}
if it is classical with a nondegenerate invariant supersymmetric even bilinear form. 

The superanalogs for reductive groups will be supergroup schemes that arise from classical ``simple'' Lie superalgebras. As far as the authors know, there has not been a 
formal theory developed for reductive supergroup schemes. The classical ``simple" Lie superalgebras  are not always simple, but are close enough to being simple. 
These Lie superalgebras have appeared frequently in the literature and are of general interest:  
\begin{itemize} 
\item $\mathfrak{gl}(m|n)$, $\mathfrak{sl}(m|n)$, $\mathfrak{psl}(n|n)$\ \ \  [{\em general and special linear Lie superalgebras}],
\item $\mathfrak{osp}(m,n)$\ \ \ [{\em ortho-symplectic Lie superalgebras}],
\item $D(2,1,\alpha)$, $F(4)$, $G(3)$\ \ \ [{\em exceptional Lie superalgebras}],
\item $\mathfrak{q}(n)$, $\mathfrak{psq}(n)$\ \ \ [{\em queer Lie superalgebras}],
\item ${\mathfrak p}(n)$, $\widetilde{{\mathfrak p}}(n)$\ \ \ [{\em periplectic Lie superalgebras}].
\end{itemize} 
For the queer Lie superalgebras, $\mathfrak{q}(n)$ will be the Lie superalgebra with even and odd parts $\mathfrak{gl}_n$, while $\mathfrak{psq}(n)$ is the corresponding simple subquotient of $\mathfrak{q}(n)$ (cf. \cite{PS}). The periplectic Lie superalgebras include ${\mathfrak p}(n)$ with even component $\mathfrak{sl}_{n}$ and its enlargement $\widetilde{{\mathfrak p}}(n)$ with  
even component $\mathfrak{gl}_n$.

Let $U({\mathfrak g})$ be the universal enveloping superalgebra of ${\mathfrak g}$. Supermodules over Lie 
superalgebras can be viewed as unital module for $U({\mathfrak g})$. If $M$ and $N$ are ${\mathfrak g}$-modules
one can employ the properties of $U({\mathfrak g})$ as a super Hopf algebra to define a ${\mathfrak g}$-module
structure on the dual $M^{*}$ and the tensor product $M\otimes N$. 

The cohomology theory of ${\mathfrak g}$-modules has a natural interpretation when one uses 
relative cohomology. The projective objects are relatively projective $U({\mathfrak g}_{\0})$-module, and every $U({\mathfrak g})$-module 
admits a relatively projective $U({\mathfrak g}_{\0})$-resolution. By using these facts, given ${\mathfrak g}$-modules, 
$M, N$, one can define the relative extension groups $\text{Ext}^{n}_{({\mathfrak g},{\mathfrak g}_{\0})}(M,N)$ by taking a relatively projective $U({\mathfrak g}_{\0})$-resolution for $M$. 
Furthermore, 
$$\text{Ext}^{n}_{({\mathfrak g},{\mathfrak g}_{\0})}(M,N)\cong \text{H}^{n}({\mathfrak g},{\mathfrak g}_{\0}; M^{*}\otimes N)$$
where $\text{H}^{n}({\mathfrak g},{\mathfrak g}_{\0}; M^{*}\otimes N)$ denotes relative Lie superalgebra cohomology which can be computed using an explicit complex (cf. 
\cite[3.1.8 Corollary, 3.1.15 Remark]{Kum},  \cite[Section 2.3]{BKN1}).

\subsection{Rational modules} Let $G$ be an affine supergroup scheme over ${\mathbb C}$ and $\text{Mod}(G)$ 
be the category of rational modules for $G$. Let $H$ be a closed subgroup scheme of $G$ and $R^{j}\text{ind}_{H}^{G}(-)$ be the higher right derived 
functors of the induction functor $\text{ind}_{H}^{G}(-)$. For a general overview about supergroup schemes and induction, the reader is referred to work 
of Brundan and Kleshchev. See \cite[Sections 2,4,5]{BruKl} \cite[Section 2]{Bru}. 

In the case when ${\mathfrak g}$ is a classical Lie superalgebra and ${\mathfrak g}=\text{Lie }G$, the category $\text{Mod}(G)$ is equivalent to locally finite integral 
modules for $\text{Dist}(G)=U({\mathfrak g})$ (cf. \cite[Corollary 5.7]{BruKl}). In particular, if 
${\mathfrak g}$ is a classical Lie superalgebra, then $\text{Mod}(G)$ is equivalent to ${\mathcal C}_{({\mathfrak g},{\mathfrak g}_{\0})}$ 
(i.e., the category of ${\mathfrak g}$-supermodules that are completely reducible over ${\mathfrak g}_{\0}$). The projectives in the category ${\mathcal C}_{({\mathfrak g},{\mathfrak g}_{\0})}$ 
are relatively projective $U({\mathfrak g}_{\0})$-modules. Therefore, if $M$ and $N$ are rational $G$-modules then 
$$\Ext^{n}_{G}(M,N)\cong \text{Ext}^{n}_{({\mathfrak g},{\mathfrak g}_{\0})}(M,N)$$
for all $n\geq 0$. 

\subsection{Parabolic subalgebras}\label{SS:parabolicsub} Let $\mathfrak{g}$ be a classical simple Lie superalgebra, and ${\mathfrak t}$ be a fixed maximal torus in ${\mathfrak g}_{\0}$. 
One can use the action of ${\mathfrak t}$ on ${\mathfrak g}$ to obtain a set of roots $\Phi$. We can now invoke the ideas presented by Grantcharov and Yakimov in \cite{GY} to 
define a parabolic subsets $S$ that correspond to parabolic subalgebras ${\mathfrak p}$ of ${\mathfrak g}$. For precise details, see \cite{GY}, \cite[3.1, 3.2]{GGNW}. 
Given a parabolic subalgebra ${\mathfrak p}$, one has a decomposition of $S=S_{0}\sqcup S^{-}$ with ${\mathfrak p}={\mathfrak l}\oplus {\mathfrak u}$ where 
(i) ${\mathfrak l}$ is the \emph{Levi subalgebra} with roots in $S_{0}$, and (ii) ${\mathfrak u}$ is the \emph{nilradical} of $\mathfrak{p}$ with roots in $S^{-}$. 

A parabolic subalgebra, ${\mathfrak b}$, is one that arises from taking a principal parabolic subset given by 
$S = S(\mathcal H) = S^0 \sqcup S^-$, where $\mathcal H$ is listed in \cite[Table 7.1.2]{GGNW}. In this case, 
${\mathfrak b}\cong {\mathfrak f}\oplus {\mathfrak u}$ where the Levi subalgebra is the detecting subalgebra ${\mathfrak f}$ that was first introduced in 
\cite{BKN1}. The Lie subalgebras ${\mathfrak f}$ (resp. ${\mathfrak u}$) are given in \cite[Table 7.1.1]{GGNW} (resp. \cite[Table 7.1.3]{GGNW}). 

The parabolic subalgebra ${\mathfrak b}$ is a parabolic subalgebra and technically is not a Borel subalgebra. In this paper, we will view ${\mathfrak b}$ as 
analogous to a Borel subalgebra for a simple Lie algebra arising from an algebraic group. The detecting subalgebra ${\mathfrak f}$ will be analogous to a maximal torus. 
There exists a natural triangular decomposition of ${\mathfrak g}={\mathfrak u}^{+}\oplus {\mathfrak f}\oplus {\mathfrak u}$ where the roots in ${\mathfrak u}^{+}$ (resp. ${\mathfrak u}$) coincide with $-(S^{-})$ (resp. $S^{-}$). Note the BBW parabolic subalgebra identifies with $\mathfrak{b} = \mathfrak{f} \oplus \mathfrak{u}$, and the BBW parabolic 
subalgebras are defined for classical simple Lie superalgebra that are not of type $P$. 

\begin{example} Consider the case when $G=GL(n|n)$ with ${\mathfrak g}=\mathfrak{gl}(n|n)$ or $G=Q(n)$ with ${\mathfrak g}={\mathfrak q}(n)$. The BBW parabolic ${\mathfrak b}$ can be realized in ${\mathfrak g}$ as 
 \[ 
{\mathfrak b}=\left\{  \left[
\begin{array}{c|c}
A  & B \\ \hline
 C & D
\end{array}\right] \in {\mathfrak g}:  A,\ B,\ C,\ D\in L_{n}({\mathbb C}) \right\}
 \]
where $L_{n}({\mathbb C})$ are the set of $n\times n$ lower triangular matrices. 
There exists a supergroup scheme $B$ with $\text{Lie }B={\mathfrak b}$ that corresponds to $\text{Dist}(B)=U({\mathfrak b})$.

In \cite[Theorem 4.10.1]{GGNW}, the sheaf cohomology $R^{\bullet}\text{ind}_{B}^{G} {\mathbb C}$ was completely described as a $G$-module. Its Poincar\'e series was shown to be equal to the 
Poincar\'e series of a finite reflection group specialized either at $s=t$ for $Q(n)$ or $s=t^{2}$ for $GL(n|n)$. 
\end{example} 

\section{Spectral sequence constructions} 

\subsection{Spectral Sequence I} Let $G$ be a supergroup scheme and $H$ be a closed subgroup scheme in $G$. Given an $H$-module, $M$, a natural question to ask is whether one can express $R^{\bullet}\text{ind}_{H}^{G} M$ when considered as a $G_{\0}$-module in terms of $R^{\bullet}\text{ind}_{H_{\0}}^{G_{\0}}(-)$. In \cite[Corollary 2.8]{Bru}, Brundan showed that this can be accomplished in the Grothendieck group of $G_{\0}$-supermodules by looking at alternating sums via Euler characters. This presents some difficulties when one wants to 
analyze $R^{n}\text{ind}_{H}^{G} M$ for a fixed $n$. The following theorem relates  $R^{n}\text{ind}_{H}^{G} M$ for a fixed $n$ as a $G_{0}$-module to certain cohomology groups for 
$R^{\bullet}\text{ind}_{H_{\0}}^{G_{\0}}(-)$ via a spectral sequence. Our construction was inspired by the result stated for the structure sheaf by Sam and Snowden (cf. 
\cite[Proposition 2.1]{sam-snowden}), and employs the work in \cite[Section 2]{Bru}. 

\begin{theorem}\label{T:spectralseq1} Let $G$ be a supergroup scheme and $H$ be a closed subgroup scheme of $G$, with ${\mathfrak g}=\operatorname{Lie }{\mathfrak g}$ and
${\mathfrak h}=\operatorname{Lie }{\mathfrak h}$. If $M$ be a $H$-module then there exists a spectral sequence 
$$E_{1}^{i,j}=R^{i+j}\operatorname{ind}_{H_{\0}}^{G_{\0}} [M\otimes \Lambda^{j}({\mathfrak g}_{\1}/{\mathfrak h}_{\1})^{*}] \Rightarrow [R^{i+j}\operatorname{ind}_{H}^{G}M]|_{G_{\0}}.$$ 
\end{theorem} 

\begin{proof} We will apply the spectral sequence construction given in \cite[E9 Theorem, Appendix E]{Kum}. In order to do so we need to construct a 
convergent cochain filtration, $F$, bounded above on the cochain complex, $C$, whose cohomology is $[R^{\bullet}\operatorname{ind}_{H}^{G}M]|_{G_{\0}}$. This will yield a 
convergent spectral sequence where $E_{1}^{i,j}=\operatorname{H}^{i+j}(F^{i}C/F^{i+1}C)$. 

Recall that $R^{\bullet}\text{ind}_{H}^{G} M=\text{H}^{\bullet}(H,M\otimes k[G])=\text{Ext}^{\bullet}_{H}(k,M\otimes k[G])$ (cf. \cite[Section 1]{FP}). Let 
$$0\rightarrow M \rightarrow I_{0}\rightarrow I_{1} \rightarrow \dots.$$ 
be an injective $H$-resolution of $M$. By tensoring by $k[G]$ one has an injective $H$-resolution for $M\otimes k[G]$:
$$0\rightarrow M\otimes k[G] \rightarrow I_{0}\otimes k[G]\rightarrow I_{1}\otimes k[G]\rightarrow \dots.$$ 
Now one filters $k[G]$ by powers of ${\mathcal I}=k[G]k[G]_{\1}$. Note that ${\mathcal I}$ is a $H$-$G_{\0}$-bimodule. 

This induces a filtration on $C^{n}=\text{H}^{0}(H, I_{n}\otimes k[G])$: 
\begin{equation}
C^{n}\supseteq \text{H}^{0}(H, I_{n}\otimes {\mathcal I}) \supseteq \text{H}^{0}(H, I_{n}\otimes {\mathcal I}^{2})\supseteq \dots.
\end{equation} 
Since $I_{n}$ is injective, $\text{H}^{1}(H, I_{n}\otimes {\mathcal I}^{k})=0$, thus 
$$\text{H}^{0}(H, I_{n}\otimes {\mathcal I}^{k}/{\mathcal I}^{k+1}) \cong 
\text{H}^{0}(H, I_{n}\otimes {\mathcal I}^{k}/I_{n}\otimes {\mathcal I}^{k+1}) \cong \text{H}^{0}(H, I_{n}\otimes {\mathcal I}^{k})/\text{H}^{0}(H, I_{n}\otimes {\mathcal I}^{k+1}).$$ 

By applying the construction described in the first paragraph, there exists a spectral sequence 
$$E_{1}^{i,j}=\text{H}^{i+j}(H,M\otimes {\mathcal I}^{i}/{\mathcal I}^{i+1})\Rightarrow \text{H}^{i+j}(H,M\otimes k[G]).$$ 
The result now follows by applying the isomorphisms given in \cite[Theorem 2.7]{Bru}
$$\text{H}^{s}(H,M\otimes {\mathcal I}^{t}/{\mathcal I}^{t+1})\cong \text{H}^{s}(H_{\0}, M\otimes \Lambda^{t}(({\mathfrak g}_{\1}/{\mathfrak h}_{\1})^{*})\otimes k[G_{\0}]).$$ 
\end{proof} 

One of the immediate consequences of this spectral sequence is the following fact. Let $G$ be a supergroup scheme arising from a classical Lie superalgebra. In this case 
$G_{\0}$ is reductive. Let $P$ be a parabolic subgroup scheme which implies that $G_{\0}/P_{\0}$ is a projective variety. Then $R^{n}\text{ind}_{P_{\0}}^{G_{\0}}$ takes 
finite-dimensional rational $P_{\0}$-modules to finite-dimensional rational $G_{\0}$-module. It now follows from Theorem~\ref{T:spectralseq1} and the fact that $R^{n}\text{ind}_{P_{\0}}^{G_{\0}}(-)=0$ for $n>\dim 
G_{\0}/ P_{\0}$ that if 
$M$ is a finite-dimensional rational $P$-module then $R^{n}\text{ind}_{P}^{G} M$ is a finite-dimensional rational $G$-module for all $n\geq 0$.

\subsection{Applications} In this section we demonstrate how several key results in \cite[Propositions 4.1.1 and 4.1.2]{GGNW} can be steamlined with shorter and more efficient proofs by using the the spectral sequence in Theorem~\ref{T:spectralseq1}\footnote{In the original statement of \cite[Proposition 4.1.1]{GGNW}, $i$ is used instead of $j$. In Corollary~\ref{C:Gzeroiso}, we use $j$ to facilitate a smoother transition 
between the notation used in the spectral sequence given in Theorem~\ref{T:spectralseq1}.}.

\begin{corollary} \label{C:Gzeroiso} Let ${\mathfrak g}=\operatorname{Lie }G$ be a classical simple Lie superalgebra and $P$ be a parabolic 
subgroup with $M$ a $P$-module. 
\begin{itemize} 
\item[(a)] Assume that 
$R^{n}\operatorname{ind}_{P_{\0}}^{G_{\0}} [M \otimes \Lambda^{i}(({\mathfrak g}_{\1}/{\mathfrak p}_{\1})^{*})]=0$  when $n\neq j$. 
Then 
$$(R^{n}\operatorname{ind}_{P}^{G} M)|_{G_{\0}}\cong R^{n}\operatorname{ind}_{P_{\0}}^{G_{\0}} [M \otimes \Lambda^{\bullet}(({\mathfrak g}_{\1}/{\mathfrak p}_{\1})^{*})]$$
for $n\geq 0$. 
\item[(b)] Assume that $M\cong {\mathbb C}$ and 
$R^{n}\operatorname{ind}_{P_{\0}}^{G_{\0}} [\Lambda^{i}(({\mathfrak g}_{\1}/{\mathfrak p}_{\1})^{*})]=0$ for $n\neq j$. Then 
$$(R^{n}\operatorname{ind}_{P}^{G} {\mathbb C})|_{G_{\0}}\cong R^{n}\operatorname{ind}_{P_{\0}}^{G_{\0}} [\Lambda^{\bullet}(({\mathfrak g}_{\1}/{\mathfrak p}_{\1})^{*})]$$
for $n\geq 0$. 
\end{itemize} 
\end{corollary} 
\begin{proof} Observe that part (b) which is \cite[Proposition 4.1.1(b)]{GGNW} follows immediately from part (a). Also, note that part (a) is a stronger version of 
\cite[Proposition 4.1.1(a)]{GGNW}.

For part (a), set $H=P_{\0}$ and apply the spectral sequence given in Theorem~\ref{T:spectralseq1}. Under the assumption, one has $E_{1}^{i,j}=0$ when $i+j\neq j$ or 
equivalently $E_{1}^{i,j}=0$ unless $i=0$. The spectral sequence lives on the vertical axis (i.e., $E_{1}^{0,j}$ for $j\geq 0$). Using the fact that the bidgrees of $d_{r}$ are $(r,1-r)$ 
(cf. \cite[E.9 Theorem, proof]{Kum}), it follows that the spectral sequence collapses and yields the isomorphism. 
\end{proof}

\begin{corollary} \label{C:Gzeroiso2} Let ${\mathfrak g}=\operatorname{Lie }G$ be a classical simple Lie superalgebra and $P$ be a parabolic 
subgroup with $M$ a $P$-module.  Assume that 
$R^{n}\operatorname{ind}_{P_{\0}}^{G_{\0}} [M \otimes \Lambda^{\bullet}(({\mathfrak g}_{\1}/{\mathfrak p}_{\1})^{*})]=0$ for $n> 0$. 
Then 
$$(R^{n}\operatorname{ind}_{P}^{G} M)|_{G_{\0}}\cong R^{n}\operatorname{ind}_{P_{\0}}^{G_{\0}} [M \otimes \Lambda^{\bullet}(({\mathfrak g}_{\1}/{\mathfrak p}_{\1})^{*})]$$
for $n\geq 0$. 
\end{corollary} 

\begin{proof} Set $H=P_{\0}$ and apply the spectral sequence given in Theorem~\ref{T:spectralseq1}. In this case, one has $E_{1}^{i,j}=0$ unless $i+j=0$ or $j=-i$. The spectral 
sequence collapses because the bidegrees of $d_{r}$ are $(r,1-r)$ for $r\geq 1$, and the result follows. 
\end{proof}

\subsection{Spectral Sequence II} One can use the theorem in \cite[I. 4.1 Proposition]{Jan} to construct a spectral sequence that relates the composition of two induction functors. 

\begin{theorem}\label{T:spectralseq2} Let $G$ be a supergroup scheme with $H\leq K \leq G$ an inclusion of closed subgroup schemes contained in $G$. If $N$ is a 
$H$-module then there exists a first quadrant spectral sequence 
$$E_{2}^{i,j}=R^{i}\operatorname{ind}_{K}^{G} [R^{j}\operatorname{ind}_{H}^{K} N] \Rightarrow R^{i+j}\operatorname{ind}_{H}^{G} N.$$ 
\end{theorem}

\subsection{Spectral Sequence III}
The third spectral sequence below was constucted in \cite[Proposition 6.2.1]{GGNW} and relates the relative Lie superalgebra cohomology with sheaf cohomology. 
The standard construction involves a composition of left exact functors. This spectral sequence is a first quadrant spectral sequence and the differentials also have bidegree 
$(r,1-r)$. This spectral sequence can be viewed analogously to the one relating cohomology for algebraic groups and sheaf cohomology presented in \cite[I.4.5 Proposition]{Jan}. 

\begin{theorem}\label{T:spectralseq3} Let $G$ be a supergroup scheme where ${\mathfrak g}$ is a classical simple Lie superalgebra, and $H$ be a closed subgroup scheme 
of $G$ with ${\mathfrak h}=\operatorname{Lie }H$. If  $M_{1}$ is a $G$-module and $M_{2}$ is a $H$-module then there exists a first quadrant spectral sequence. 
$$E_{2}^{i,j}=\operatorname{Ext}^{i}_{({\mathfrak g},{\mathfrak g}_{\0})} (M_{1},R^{j}\operatorname{ind}_{H}^{G} M_{2})\Rightarrow 
\operatorname{Ext}_{({\mathfrak h},{\mathfrak h}_{\0})}^{i+j}(M_{1},M_{2}). $$ 
\end{theorem} 

\section{Irreducible Representations via $H^{0}(\lambda)$}\label{S:simplemodules}

\subsection{} Throughout this section, we will assume that $G$ is a classical simple algebraic supergroup scheme and $B$ is a BBW parabolic for $G$. 
In particular, we will be tacitly assuming that $G$ is not of type $P$. Recall that one has a triangular decomposition ${\mathfrak g}={\mathfrak u}\oplus {\mathfrak f}\oplus {\mathfrak u}^{+}$ 
with corresponding supergroup schemes $G$, $U$, $F$ and $U^{+}$. The supergroup schemes $U$ and $U^{+}$ are unipotent, that is, the only finite-dimensional simple module for these subgroup schemes is ${\mathbb C}$. 

There exists a maximal torus $T_{\0}$ contained in the even part of $F$, and set $X=X(T_{\0})$. For a given $F$, there exists a subset of weights $X_{F}\subseteq X$ that 
indexes the set of irreducible representations for $F$. For $\lambda\in X_{F}$, let $L_{\mathfrak f}(\lambda)$ be the corresponding simple $F$-module. One can then inflate this module to $B=F\ltimes U$, and consider $H^{n}(\lambda)=R^{n}\text{ind}_{B}^{G} L_{\mathfrak f}(\lambda)$. 

The goal of this section is to show how to classify finite-dimensional simple $G$-module via $G$-socles of $H^{0}(\lambda)$. The proofs follows along the same lines are in \cite[II Chapter 2] {Jan} 
\cite[Section 6]{BruKl}, and generalize the results for $Q(n)$ stated in \cite[Theorem 4.4]{Bru2}.

\subsection{Simple $F$-modules} For the algebraic supergroup scheme $F$ where $\text{Lie }F={\mathfrak f}$ is a detecting subalgebra, one can determine the set $X_{F}$. 

\begin{example} Let $G=Q(n)$. In this case $F\cong Q(1)\times Q(1)\times \dots \times Q(1)$, and $X_{F}=X(T_{\0})$. The irreducible representations are given by modules 
$u(\lambda)$ \cite[Lemma 6.4]{BruKl}. 
\end{example}

\begin{example} Let $G=GL(n|n)$. The subgroup $F\cong GL(1|1)\times GL(1|1)\times \dots GL(1|1)$ and $X_{F}=X(T_{\0})$. The irreducible representations are formed by taking outer tensor products of simple $GL(1|1)$-representations which are either one-dimensional or two-dimensional. 
\end{example} 

\begin{example} Let $G=GL(m,n)$. Without loss of generality we may assume that $m\leq n$. In this case, $F\cong [GL(1|1)\times GL(1|1)\times \dots GL(1|1)]\times GL_{n-m}$ where there are 
$m$ copies of $GL(1|1)$. Moreover, $X_{F}=X(T_{1})\times X(T_{2})_{+}$ where $T_{1}$ is a maximal torus for $GL(1|1)\times GL(1|1)\times \dots GL(1|1)$, and $T_{2}$ is a maximal torus for 
$GL_{n-m}$. Note that one must take dominant weights on $X(T_{2})$, and $X_{F}$ is a proper subset of $X$. 
\end{example}

\subsection{} Let $L$ be a finite-dimensional simple $G$-module. Then for some $\lambda\in X_{F}$, $\text{Hom}_{B}(L,L_{\mathfrak f}(\lambda))\neq 0$. Therefore, by 
Frobenius reciprocity $0\neq \text{Hom}_{G}(L,H^{0}(\lambda))$, and $L\hookrightarrow H^{0}(\lambda)$ for some $\lambda\in X_{F}$. Let $X_{F,+}=\{\lambda\in X_{F}:\ H^{0}(\lambda)\neq 0\}$. 

\begin{proposition} Let $\lambda\in X_{F,+}$. Then $H^{0}(\lambda)^{U^{+}}\cong L_{\mathfrak f}(\lambda)$. 
\end{proposition} 

\begin{proof} We first consider a more general idea about induction. Let $M$ be a rational $B$-module and let $\epsilon_{M}:\text{ind}_{B}^{G}M\rightarrow M$ be the evaluation 
homomorphism. Using the same proof in \cite[II Proposition 2.2 and (1)]{Jan}, one can show that $[\text{ind}_{B}^{G}M]^{U^{+}}\hookrightarrow M$ under $\epsilon_{M}$. This is a monomorphism of 
$F$-modules. 

Now apply this to the case when $M=L_{\mathfrak f}(\lambda)$. The statement of the proposition now follows since $L_{\mathfrak f}(\lambda)$ is simple as an $F$-module
and the $U^{+}$-fixed points $H^{0}(\lambda)$ cannot be zero for $\lambda\in X_{F,+}$. 
\end{proof}

\subsection{} We can now give a parametrization of simple $G$-modules.

\begin{theorem} Let $G$ be a classical simple algebraic subgroup scheme. Then 
there is a $1$-$1$ correspondence between simple $G$-modules and $X_{F,+}$ given by 
$L(\lambda)=\operatorname{soc}_{G}H^{0}(\lambda)$. 
\end{theorem} 

\begin{proof} First we need to show that if $\lambda\in X_{F,+}$ then $\text{soc}_{G}H^{0}(\lambda)$ is simple. This can easily be seen because if 
$L_{1}$ and $L_{2}$ are simple $G$-modules with $L_{1}\oplus L_{2}\hookrightarrow H^{0}(\lambda)$, then one can take $U^{+}$-fixed points to get 
a monomorphism of $F$-modules: $L_{1}^{U^{+}}\oplus L_{2}^{U^{+}}\hookrightarrow L_{\mathfrak f}(\lambda)$. Since $U_{+}$-fixed points on 
$L_{j}$, $j=1,2$ are non-trivial and $L_{\mathfrak f}(\lambda)$ is a simple $F$-module, one obtains a contradiction. 

Let $L=\text{soc}_{G}H^{0}(\lambda)$. Then $L^{U^{+}}\cong L_{\mathfrak f}(\lambda)$. This shows that the socles of $H^{0}(\lambda)$ and $H^{0}(\mu)$ where 
$\lambda,\mu\in X_{+}$ are not isomorphic unless $\lambda=\mu$. Therefore, for $\lambda\in X_{+}$, one can set $L(\lambda)=\text{soc}_{G}H^{0}(\lambda)$ to 
obtain the desired bijective correspondence.  
\end{proof} 

\subsection{} We will now indicate how one can parametrize the simple modules using this setup for $G=Q(n)$. Let $M$ be a $G$-module and $M=\oplus_{\mu \in X}M_{\mu}$ be its 
weight space decomposition. We have ${\mathfrak f}={\mathfrak f}_{\0}\oplus {\mathfrak f}_{\1}$ with ${\mathfrak f}_{\0}\cong {\mathfrak t}$, and $[{\mathfrak f}_{\0},{\mathfrak f}_{\1}]=0$. 
This implies that the weight space $M_{\mu}$ is an $F$-module. 

Now let $M$ be a simple $Q(n)$-module. Then for some $\lambda\in X_{F}$, 
$$0\neq \text{Hom}_{G}(M,\text{ind}_{B}^{G} L_{\mathfrak f}(\lambda))=\text{Hom}_{B}(M,L_{\mathfrak f}(\lambda)).$$ 
It follows that $L_{\mathfrak f}({\lambda})$ has to appear in the head of $M$ as $B$-module and $\lambda$ must be the highest weight of $M$. The ordering is given by the 
roots $\Delta_{\0}=\{\epsilon_{1}-\epsilon_{2},\epsilon_{2}-\epsilon_{3},\dots,\epsilon_{n-1}-\epsilon_{n}\}$. Furthermore, 
$$0\neq\text{Hom}_{B}(M,L_{\mathfrak f}(\lambda))\subseteq \text{Hom}_{B_{\0}}(M,L_{\mathfrak f}(\lambda))=\text{Hom}_{G_{\0}}(M,\text{ind}_{B_{\0}}^{G_{\0}}L_{\mathfrak f}(\lambda)).$$ 
Now $L_{\mathfrak f}(\lambda)=\oplus \lambda$ as a $B_{\0}$-module, so it follows that $\lambda$ must be in $(X_{\0})_{+}$ (i.e., it is a dominant integral weight). The upshot of this analysis is that 
$L(\lambda)=\text{soc}_{G}H^{0}(\lambda)$ where $\lambda\in (X_{\0})_{+}$ and $\lambda$ is the highest weight of $L(\lambda)$ (cf. \cite[Theorem 4.18]{Bru2}). 

For $G$ not of type $Q$ the weight spaces no longer yield $F$-modules, so this analysis will not work. An interesting problem would be to provide explicit parameterizations of simple modules involving weights 
for the other classical simple Lie superalgebras. Moreover, once one has an explicit parametrization, an interesting problem would be to develop a theory of  decomposition numbers (e.g., 
$[H^{0}(\lambda):L(\mu)]$ for $\lambda,\mu \in X_{F,+}$).  

\section{Generic Behavior for BBW Parabolics} 

\subsection{Redux: GGNW Computations} Assume throughout this section that ${\mathfrak g}$ is a classical simple Lie superalgebra not of type $P$. Furthermore, let $G$ be a supergroup scheme with ${\mathfrak g}=\text{Lie }G$, and $B$ a BBW parabolic subgroup of $G$. Set  
$$p_{G,B}(t)=\sum_{i=0}^{\infty} \dim R^{i}\text{ind}_{B}^{G} {\mathbb C}\ t^{i}.$$
For the detecting subalgebra ${\mathfrak f}$ associated to ${\mathfrak b}$, there is an isomorphism of rings given by the restriction map:
$$S^{\bullet}({\mathfrak g}_{\1}^{*})^{G_{\0}}\cong S^{\bullet}({\mathfrak f}_{\1}^{*})^{N}.$$
where $N$ is a reductive algebraic group. If $N_{0}$ is the connected component of the identity in $N$ then $W_{\1}=N/N_{0}$ is a finite reflection group. 
Let $p_{W_{\1}}(s)=\sum_{w\in W_{\1}}s^{l(w)}$ be the Poincar\'e polynomial for $W_{\1}$. 

A fundamental result in \cite[Sections 4.2-4.9]{GGNW} was the calculation of $R^{\bullet}\text{ind}_{B_{\0}}^{G_{\0}} \Lambda^{\bullet}(({\mathfrak g}_{\1}/{\mathfrak b}_{\1})^{*})$. It was shown that 
\begin{equation}
R^{n}\text{ind}_{B_{\0}}^{G_{\0}} \Lambda^{j}(({\mathfrak g}_{\1}/{\mathfrak b}_{\1})^{*})=0\ \ \text{for $n\neq j$.}
\end{equation}
Furthermore, in the case when $n=j$, $R^{n}\text{ind}_{B_{\0}}^{G_{\0}} \Lambda^{n}(({\mathfrak g}_{\1}/{\mathfrak b}_{\1})^{*})$ is a direct sum of 
trivial modules whose number is prescribed by $p_{W_{\1}}(s)$. These results in conjunction with Corollary~\ref{C:Gzeroiso} yield the calculation of  
$R^{\bullet}\text{ind}_{B}^{G} {\mathbb C}$ which is summarized below (cf. \cite[Theorem 4.10.1]{GGNW}).

\begin{theorem} \label{T:GGNWRedux} Let ${\mathfrak g}$ be a classical simple Lie superalgebra with ${\mathfrak g}=\operatorname{Lie }G$. 
Assume that ${\mathfrak g}$ is not isomorphic to $P(n)$. Let $B$ be the parabolic subgroup such that ${\mathfrak b}=\operatorname{Lie }B$ where ${\mathfrak b}$ is a 
BBW parabolic subalgebra. Then 
\begin{itemize} 
\item[(a)] $R^{\bullet}\operatorname{ind}_{B}^{G}{\mathbb C}$ is a direct sum of trivial modules. 
\item[(b)] The number of trivial modules in $R^{n}\operatorname{ind}_{B}^{G}{\mathbb C}$ is given by 
$$p_{G,B}(t)=p_{W_{\1}}(s)$$ 
where $s=t$ for $G$ is of type $Q$, and $s=t^{2}$ otherwise.  
\end{itemize} 
\end{theorem}

\subsection{An Analog of Kempf's Vanishing Theorem} Let $T_{\0}$ be a maximal torus in $G_{\0}$, $X=X(T_{\0})$ and $(X_{\0})_{+}$ the dominant integral weights. 
The Weyl group of $G_{\0}$ is denoted by $W_{\0}$ with identity element $1\in W_{\0}$. 

Moreover, let $V$ be a $T_{\0}$-module and $V=\oplus_{\gamma\in X} V_{\gamma}$ be its weight space decomposition. Set $\text{wt}(V)=\{\gamma\in X:\ V_{\gamma}\neq 0\}$ (i.e., the set of weights of $V$). We start off this section by stating a key definition. 

\begin{definition} Let $\lambda\in X$ and $w\in W_{\0}$. 
\begin{itemize} 
\item[(a)] The weight $\lambda$ is {\em very dominant} if $\mu+\sigma\in X_{+}$ for all $\mu\in \operatorname{wt}(L_{\mathfrak f}(\lambda))$ and 
$\sigma\in \operatorname{wt}(\Lambda^{\bullet}(({\mathfrak g}/{\mathfrak b})^{*}))$. 
\item[(b)] The set of very dominant weight will be denoted by $X_{++}$.
\item[(c)] Set $\Gamma(\lambda,w)=\operatorname{wt}(L_{\mathfrak f}(\lambda)\otimes w^{-1}\Lambda^{\bullet}(({\mathfrak g}/{\mathfrak b})^{*}))$. For 
$\gamma\in \Gamma(\lambda,w)$, let $m_{\gamma,w}$ be the multiplicity of the weight $\gamma$ in $L_{\mathfrak f}(\lambda)\otimes w^{-1}\Lambda^{\bullet}(({\mathfrak g}/{\mathfrak b})^{*})$. 
\end{itemize} 
\end{definition} 
 As a consequence of Theorem~\ref{T:spectralseq1}, we can provide a criterion for the vanishing of the higher sheaf cohomology groups for weights that are 
very dominant. 

\begin{theorem}  \label{T:Kempf} Let $\lambda\in X_{++}$, and $1$ be the identity element in $W_{\0}$. Then 
\begin{itemize}
\item[(a)] $R^{n}\operatorname{ind}_{B}^{G} L_{\mathfrak f}(\lambda)=0$ for $n>0$. 
\item[(b)] $\operatorname{ind}_{B}^{G} L_{\mathfrak f}(\lambda)|_{G_{\0}}\cong \oplus_{\gamma\in \Gamma(\lambda,1)} [\operatorname{ind}_{B_{\0}}^{G_{\0}} \gamma]^{\oplus m_{\gamma,1}}$ as a $G_{\0}$-module. 
\end{itemize}
\end{theorem} 

\begin{proof} One can apply the spectral sequence in Theorem~\ref{T:spectralseq1} with $H=B$ and $M=L_{\mathfrak f}(\lambda)$. Under the condition that 
$\lambda\in X_{++}$, one has $E_{1}^{i,j}=0$ for $i+j>0$. Therefore, the spectral sequence degenerates and yields part (a). Part (b) follows because under the 
assumption that $\lambda\in X_{++}$, one has $R^{1}\text{ind}_{B_{\0}}^{G_{\0}} \gamma=0$ for all $\gamma\in \Gamma(\lambda,1)$. 
\end{proof} 

We can now illustrate how this theorem works for ${\mathfrak q}(n)$. 

\begin{example} Let ${\mathfrak g}={\mathfrak q}(n)$, $G=Q(n)$ and $B$ be a BBW parabolic subgroup. For $\lambda\in X$, 
$L_{\mathfrak f}(\lambda)\cong \lambda^{\oplus \dim L_{\mathfrak f}(\lambda)}$ (direct sum of copies of $\lambda$) as a $B_{\0}$-module.

First observe that $\lambda=0$ is not very dominant because $0\neq R^{1}\text{ind}_{B}^{G}{\mathbb C}=R^{1}\text{ind}_{B}^{G} \lambda$ by 
Theorem~\ref{T:GGNWRedux}. Let $\lambda\in X_{++}$. In this case, $\lambda\in X_{++}$ if and only if $\lambda+\sigma\in (X_{0})_{+}$ for all $\sigma\in \operatorname{wt}(\Lambda^{\bullet}(({\mathfrak g}/{\mathfrak b})^{*}))$. Since $\sigma$ can be zero, one has $X_{++}\subseteq (X_{\0})_{+}$. 

The weight $\sigma$ is a sum of distinct roots from the set $-\Phi_{\1}^{+}=\{\epsilon_{i}-\epsilon_{j}:\ 1\leq i<j \leq n\}$. Now the simple roots for $G_{\0}$ are given by 
$\Delta_{\0}=\{\alpha_{1},\dots,\alpha_{n}\}$ where $\alpha_{i}=\epsilon_{i}-\epsilon_{i+1}$ where $i=1,2,\dots,n-1$. The condition that
 $\lambda+\sigma\in (X_{0})_{+}$ is equivalent to $\langle \lambda+\sigma,\alpha^{\vee}\rangle \geq 0$ for $\alpha\in \Delta_{\0}$. 
 
 A direct calculation shows that $-\langle \sigma,\alpha^{\vee}\rangle\geq n+1$, and it follows that 
 $$\{\lambda\in X:\ n+1\leq \langle \lambda,\alpha^{\vee} \rangle\ \text{for all $\alpha\in \Delta_{\0}$}\} \subseteq X_{++} \subseteq (X_{\0})_{+}.$$ 
\end{example} 

\subsection{An Analog of the Bott-Borel-Weil Theorem} For $w\in W_{\0}$, recall that the dot action on $X$ is given by $w\cdot \lambda=w(\lambda+\rho)-\rho$. 
Let $\overline{C}_{\mathbb Z}$ for $G_{\0}$ be defined as in \cite[II. 5.5]{Jan}. 

For a given $w\in W_{\0}$, set 
$$\Omega(w)=\{\lambda\in X: \mu+w^{-1}\sigma\in \overline{C}_{\mathbb Z}\ \text{for all $\mu\in \operatorname{wt}(L_{\mathfrak f}(\lambda))$ and 
$\sigma\in \operatorname{wt}(\Lambda^{\bullet}(({\mathfrak g}_{\1}/{\mathfrak b}_{\1})^{*}))$} \}.$$
Observe that $\Omega(w)\subseteq \overline{C}_{\mathbb Z}$ for all $w\in W_{\0}$ since $0$ is a weight of $\Lambda^{\bullet}(({\mathfrak g}_{\1}/{\mathfrak b}_{\1})^{*})$. 
We say that a weight $\mu$ is a {\em generic weight} if and only if $\mu\in \cup_{w\in W_{\0}} w\cdot \Omega(w):=\Omega$. The set $\Omega$ will be called the set of 
generic weights. 

We can now prove a version of the Bott-Borel-Weil Theorem for generic weights. Note that this theorem encompasses  Theorem~\ref{T:Kempf} which coincides with how the ordinary BBW Theorem encompasses the classical Kempf's vanishing theorem (see \cite[II. Chapters 4 and 5]{Jan}).

\begin{theorem} \label{T:generic} Let $w\in W_{\0}$ and $w\cdot \lambda$ is a generic weight where $\lambda\in \Omega(w)$. Then 
\begin{equation*} 
(R^{n}\operatorname{ind}_{B}^{G} L_{\mathfrak f}(w\cdot \lambda))|_{G_{\0}} \cong \begin{cases} 
\bigoplus_{\gamma\in \Gamma(\lambda,w)} [\operatorname{ind}_{B_{\0}}^{G_{\0}} \gamma]^{\oplus m_{\gamma,w}} & \text{$n=l(w)$} \\
0 & \text{$n\neq l(w)$}.
\end{cases}
 \end{equation*} 
\end{theorem} 

\begin{proof} Let $\mu+w^{-1}\sigma\in \overline{C}_{\mathbb Z}$ where $\mu$ is a weight of $L_{\mathfrak f}(w\cdot \lambda)$ and $\sigma$ a weight of $\Lambda^{\bullet}(({\mathfrak g}_{\1}/{\mathfrak b}_{\1})^{*})$. According to the  
ordinary BBW Theorem \cite[II 5.5 Corollary]{Jan}, one has 
\begin{equation} \label{eq:BBWcalc}
R^{n}\text{ind}_{B_{\0}}^{G_{\0}} w\cdot (\mu+w^{-1}\sigma)\cong  
\begin{cases} 0 & \text{if $n\neq l(w)$} \\
\text{ind}_{B_{\0}}^{G_{\0}} \mu+w^{-1}\sigma &\text{if $n=l(w)$}. 
\end{cases} 
\end{equation} 
Now apply the spectral sequence in Theorem~\ref{T:spectralseq1} with $H=B$ and $M=L_{\mathfrak f}(w\cdot \lambda)$. From (\ref{eq:BBWcalc}), it follows that 
$E_{1}^{i,j}=0$ for $i+j\neq l(w)$. One can now apply the same reasoning as  in the proof of Theorem~\ref{T:Kempf}, The spectral sequence degenerates and yields the desired result. 
\end{proof} 

\subsection{} Let $G=Q(n)$. Since $\text{wt}(L_{\mathfrak f}(\lambda))=\{\lambda\}$, one has for a given $w\in W_{\0}$, 
$$\Omega(w)=\{\lambda\in X: \lambda+w^{-1}\sigma\in \overline{C}_{\mathbb Z}\ \text{for all $\sigma\in \operatorname{wt}(\Lambda^{\bullet}(({\mathfrak g}_{\1}/{\mathfrak b}_{\1})^{*}))$} \}.$$
We now show that $\Omega$ can be obtained by translating $\Omega(1)$ by the ordinary action of the Weyl group $W_{\0}$.

\begin{lemma}
Let $\mathfrak{g}$ be a semisimple Lie algebra and let $\sigma$ be a sum of distinct negative roots of $\mathfrak{g}$.  Then for all $w$ in the Weyl group $W$ of $\mathfrak{g}$, $w \cdot \sigma$ is also a sum of distinct negative roots.
\end{lemma}

\begin{proof}
Let $A_w$ be the set of all roots $\alpha$ in $\Phi^+$ such that $w(\alpha) \in \Phi^-$, and let $B_w$ be the set of all roots $\beta$ in $\Phi^+$ such that $w(\beta) \in \Phi^+$.  If $\sigma$ is  sum of distinct negative roots, then it may be written as
$$\sigma = -\alpha_1 - \cdots - \alpha_n - \beta_1 - \cdots - \beta_m$$
where $\alpha_i \in A_w$ for all $i$ and $\beta_j \in B_w$ for all $j$.  Then we have that
$$w \cdot \sigma = w(\sigma) + w \cdot 0 = -w(\alpha_1) - \cdots - w(\alpha_n) + w(-\beta_1) + \cdots + w(-\beta_m) + w\cdot 0.$$
Notice that
$$w \cdot 0 = \sum_{\alpha \in A_w} w(\alpha),$$
and so
$$-w(\alpha_1) - \cdots - w(\alpha_n) + w \cdot 0$$
is a sum of distinct negative roots
$$-w(\alpha_1) - \cdots - w(\alpha_n) + w \cdot 0 = w(\gamma_1) + \cdots + w(\gamma_l),$$
where the $\gamma_k$ are all in $A_w$.  Moreover, since $\beta_j$ is in $B_w$, each $w(-\beta_j)$ is a negative root, so this implies that $w \cdot \sigma$ is a sum of negative roots.  Finally, since the $\gamma_k$ and $-\beta_j$ are all distinct roots, so too are the $w(\gamma_k)$ and $w(-\beta_j)$, so $\sigma$ is a sum of distinct negative roots.
\end{proof}

\begin{proposition} Let $G=Q(n)$. The generic regions $w \cdot \Omega(w)$ are conjugate under the regular action of the Weyl group $W_{\0}$. Consequently, 
$\Omega=\cup_{w\in W_{\0}} w(\Omega(1))$.
\end{proposition}
\begin{proof}
It is enough to show that $w \cdot \Omega(w)$ is equal to $w(\Omega(1))$.  Let $\lambda \in \Omega(1)$.  Then for all positive roots $\alpha$ and all sums of distinct negative roots $\sigma$,
$$\langle \lambda + \sigma+ \rho, \alpha^{\vee} \rangle \geq 0.$$
However, by the above lemma, $w^{-1} \cdot \sigma$ is a sum of distinct negative roots, and so
$$\langle \lambda + w^{-1} \cdot \sigma +\rho, \alpha^{\vee} \rangle \geq 0.$$
Now 
$$\lambda + \rho + w^{-1} \cdot \sigma=[\lambda+ w^{-1}\rho-\rho]+w^{-1}\sigma +\rho.$$ 
Therefore, $\lambda+w^{-1}\rho-\rho$ is an element of $\Omega(w)$.  This, however, is equal to $w^{-1} \cdot (w \lambda)$, and so $w \lambda \in w \cdot \Omega(w)$.  Thus, 
$w(\Omega(1)) \subseteq w \cdot \Omega(w)$.  The other direction follows similarly.
\end{proof}

\subsection{Example: $G=Q(2)$} In this case $G_{\0}\cong GL_{2}$ and $W_{\0}\cong \Sigma_{2}=\{1,s_{\alpha} \}$. Moreover, 
$\text{wt}(\Lambda^{\bullet}(({\mathfrak g}_{\1}/{\mathfrak b}_{\1})^{*}))=\{0,-\alpha\}$. Using the definition of $\Omega(w)$, one can directly show that 
\begin{eqnarray*}
\Omega(1)&=&\{(\lambda_{1},\lambda_{2}):\ \lambda_{1}-\lambda_{2}\geq 1\} \\
\Omega(s_{\alpha})&=&\{(\lambda_{1},\lambda_{2}):\ \lambda_{1}-\lambda_{2}\geq -1\} 
\end{eqnarray*} 
Therefore, 
$$\Omega= \bigcup_{w\in W_{\0}} w\cdot \Omega(w)=\{\mu=(\mu_{1},\mu_{2}):\ \mu_{1}-\mu_{2}\neq 0\}.$$ 
It follows that for $\mu\in \Omega$, $H^{n}(\mu)|_{G_{\0}}$ can be computed for all $n$ by Theorem~\ref{T:generic}. 
This agrees with the calculation for $G=Q(2)$ given in \cite[Lemma 4.4]{Bru}. 

\subsection{Example: $G=Q(3)$} One has $G_{\0}\cong GL_{3}$ and 
$$W_{\0}\cong \Sigma_{3}=\{1,s_{\alpha_{1}},s_{\alpha_{2}},s_{\alpha_{1}}s_{\alpha_{2}},s_{\alpha_{2}}s_{\alpha_{1}},s_{\alpha_{1}}s_{\alpha_{2}}s_{\alpha_{1}} \}.$$
Moreover, the generic region $\Omega= \bigcup_{w\in W_{\0}} w\cdot \Omega(w)$ where 

\begin{eqnarray*}
\Omega(1)&=&\{\lambda\in X:\ \langle \lambda, \alpha_{1}^{\vee} \rangle \geq 2, \langle \lambda, \alpha_{2}^{\vee} \rangle \geq 2 \} \\
\Omega(s_{\alpha_{1}})&=&\{\lambda\in X:\ \langle \lambda, \alpha_{1}^{\vee} \rangle \geq 3, \langle \lambda, \alpha_{2}^{\vee} \rangle \geq 0 \}   \\
\Omega(s_{\alpha_{2}})&=&\{\lambda\in X:\ \langle \lambda, \alpha_{1}^{\vee} \rangle \geq 0, \langle \lambda, \alpha_{2}^{\vee} \rangle \geq 3 \}   \\
\Omega(s_{\alpha_{1}}s_{\alpha_{2}})&=&\{\lambda\in X:\ \langle \lambda, \alpha_{1}^{\vee} \rangle \geq 3, \langle \lambda, \alpha_{2}^{\vee} \rangle \geq -1 \}   \\
\Omega(s_{\alpha_{2}}s_{\alpha_{1}})&=&\{\lambda\in X:\ \langle \lambda, \alpha_{1}^{\vee} \rangle \geq -1, \langle \lambda, \alpha_{2}^{\vee} \rangle \geq 3 \}  \\
\Omega(s_{\alpha_{1}}s_{\alpha_{2}}s_{\alpha_{1}})&=&\{\lambda\in X:\ \langle \lambda, \alpha_{1}^{\vee} \rangle \geq 0, \langle \lambda, \alpha_{2}^{\vee} \rangle \geq 0 \}  \\
\end{eqnarray*} 
Therefore, it can be shown that
$$\Omega = \bigcup_{w\in W_{\0}} w\{\lambda\in X:\ \langle \lambda, \alpha_{1}^{\vee} \rangle \geq 2, \langle \lambda, \alpha_{2}^{\vee} \rangle \geq 2 \}, $$
i.e., the $W_{\0}$-conjugates of $\Omega(1)$ under the regular action. The generic region for this case is given below via the shaded regions. 
\vskip .5cm 
\begin{center}
\begin {tikzpicture}[baseline=.5, scale=1]
\begin {rootSystem}{A}
\wt{1}{0}
\draw \weight{0}{0} -- \weight{0}{1};
\draw[dashed] \weight{1}{0} -- \weight{9}{0};
\draw[dashed] \weight{0}{1} -- \weight{0}{9};
\node[above left=-2pt] at (hex cs:x=0,y=0) {\small\(0\)};
\node[above right=-2pt] at (hex cs:x=0,y=1) {\small\(\omega_2\)};
\node[above right=-2pt] at (hex cs:x=1,y=0) {\small\(\omega_1\)};
\draw \weight{0}{0} -- \weight{1}{0};
\wt{0}{1}
\wt{2}{2}
\wt{-2}{-2}
\wt{2}{-4}
\wt{4}{-2}
\wt{-4}{2}
\wt{-2}{4}
\fill[gray!50,opacity=.6] \weight{2}{2} -- \weight{2}{7} -- \weight{7}{2};
\fill[gray!50,opacity=.6] \weight{-2}{-2} -- \weight{-2}{-7} -- \weight{-7}{-2};
\fill[gray!50,opacity=.6] \weight{-2}{4} -- \weight{-2}{9} -- \weight{-7}{9};
\fill[gray!50,opacity=.6] \weight{2}{-4} -- \weight{2}{-9} -- \weight{7}{-9};
\fill[gray!50,opacity=.6] \weight{4}{-2} -- \weight{9}{-2} -- \weight{9}{-7};
\fill[gray!50,opacity=.6] \weight{-4}{2} -- \weight{-9}{2} -- \weight{-9}{7};
\end {rootSystem}
\end {tikzpicture}
\end{center}
\vskip .5cm 

\subsection{Comparison of Cohomology for $({\mathfrak g},{\mathfrak g}_{\0})$ and $({\mathfrak b},{\mathfrak b}_{\0})$} For reductive algebraic groups, one 
can use the induction functor to compare cohomology for $G$ to that for $P$ where $P$ is any parabolic subgroup \cite[I. 4.5 Proposition]{Jan}. Using Theorem~\ref{T:spectralseq3} and 
\ref{T:generic}, one can obtain a comparison theorem for extensions between modules for $({\mathfrak g},{\mathfrak g}_{\0})$ and $({\mathfrak b},{\mathfrak b}_{\0})$. 

\begin{theorem}\label{T:compareg-b} Let $G$ be a supergroup scheme where ${\mathfrak g}$ is a classical simple Lie superalgebra, and $B$ be a BBW parabolic subgroup of $G$. 
Moreover,  let $\lambda \in X$, $w\in W_{\0}$, $w\cdot \lambda$  be a generic weight where $\lambda\in \Omega(w)$ and $M$ be a $G$-module. Then for $i\geq 0$, 
$$\operatorname{Ext}^{i}_{({\mathfrak g},{\mathfrak g}_{\0})} (M,R^{l(w)}\operatorname{ind}_{B}^{G} L_{\mathfrak f}(w\cdot \lambda))\cong 
\operatorname{Ext}_{({\mathfrak b},{\mathfrak b}_{\0})}^{i+l(w)}(M,L_{\mathfrak f}(w\cdot\lambda)). $$ 
\end{theorem}

\begin{proof} Consider the spectral sequence in Theorem~\ref{T:spectralseq3} with $H=B$. Under the condition that  $w\cdot \lambda$  be a generic weight, 
$R^{j}\operatorname{ind}_{B}^{G} L_{\mathfrak f}(w\cdot \lambda)\neq 0$ when $j\neq l(w)$. Therefore, the spectral sequence collapses, and 
$E_{2}^{i,l(w)}\cong \operatorname{Ext}_{({\mathfrak b},{\mathfrak b}_{\0})}^{i+l(w)}(M,L_{\mathfrak f}(w\cdot\lambda))$ for all $i\geq 0$.
\end{proof} 

\subsection{Summary: Characters of $H^{\bullet}(\lambda)$} The following fundamental (open) problems are of central importance in our theory of sheaf cohomology with 
BBW parabolic subgroups. 

\begin{problem} Determine when $H^{n}(\lambda)\neq 0$. 
\end{problem} 

\begin{problem} Compute $\operatorname{ch } H^{n}(\lambda)$ for all $\lambda\in X$ and $n\geq 0$. 
\end{problem} 

The problem is equivalent to determining the composition factors of  $\operatorname{ch} H^{n}(\lambda)$ as a $G_{\0}$-module. 
Some information especially for small weights can be obtained via methods of Broer (cf. \cite[Lemma 2.10]{Br}). Brundan \cite[Corollary 2.8]{Bru} proved that the Euler characteristic is given 
by 
$$\sum_{n\geq 0} (-1)^{n} \text{ch } H^{n}(\lambda)=\sum_{n\geq 0} (-1)^{n} \text{ch } R^{n}\text{ind}_{B_{\0}}^{G_{\0}}L_{\mathfrak f}(\lambda)\otimes 
\Lambda^{\bullet}(({\mathfrak g}_{\1}/{\mathfrak b}_{\1})^{*}).$$
From our results in the previous sections, 
\begin{itemize}
\item[(i)] $\text{ch } H^{n}(0)=\text{ch } R^{n}\text{ind}_{B}^{G}{\mathbb C}$ is known by Theorem~\ref{T:GGNWRedux}. 
\item[(ii)] $\text{ch } H^{n}(\lambda)$ is known for large weights $\lambda$ given by the conditions in Theorems~\ref{T:Kempf} and~\ref{T:generic}. 
\end{itemize} 
It remains to determine the behavior for $\text{ch } H^{n}(\lambda)$ when $\lambda$ is outside  the ``generic region" $\Omega$. 

\section{Results on $H^{1}(\lambda)$}

\subsection{} For the moment, assume that $G$ is a reductive algebraic group and $B$ is a Borel subgroup (arising from the negative roots) \cite{Jan}. If $\lambda$ is a weight then 
Andersen \cite{And} proved that either $H^{1}(\lambda)=\text{ind}_{B}^{G} \lambda$ is either zero or has a simple $G$-socle. The socle of $H^{1}(\lambda)$, $\text{soc}_{G}H^{1}(\lambda)$, can be 
explicitly described \cite[II 5.15 Proposition]{Jan}. For $H^{n}(\lambda)$, $n\geq 2$, the vanishing behavior remains an open question over fields of characteristic $p>0$. 

Let us now return to the situation where $G$ is a supergroup scheme with $\text{Lie }G={\mathfrak g}$ where ${\mathfrak g}$ is a simple classical Lie superalgebra and $B$ is a BBW parabolic subgroup in 
$G$. In Section~\ref{S:simplemodules}, we proved that $H^{0}(\lambda)$ is either zero or has simple socle. In dramatic contrast to the situation for reductive groups, 
Theorem~\ref{T:GGNWRedux} demonstrates that $H^{1}(\lambda)$ need not have simple socle. For example, if $G=Q(3)$ then $H^{1}((0,0,0))=R^{1}\text{ind}_{B}^{G}{\mathbb C}\cong {\mathbb C}\oplus 
{\mathbb C}$. 

\subsection{Socles for $H^{1}(\lambda)$} Let $P=L_{P}\ltimes U_{P}$ be a parabolic subgroup such that $B\subseteq P \subseteq G$. For any $\sigma\in X_{F,+}$, let $\bar{\sigma}$ be the weight in $X$ with $L_{P}(\bar{\sigma})=L(\sigma)^{U_{P}}$ where $L_{P}(\bar{\sigma})$ is the inflation of a simple $L_{P}$-module. 

\begin{theorem}\label{T:soclecompare} Let $G$ be a supergroup arising from a simple Lie superalgebra ${\mathfrak g}$, $B$ be a BBW parabolic subgroup and $\lambda\in X$. Suppose there exists a parabolic subgroup scheme $P$ in $G$ with 
$B\subseteq P \subseteq G$ with $R^{0}\operatorname{ind}_{B}^{P} L_{\mathfrak f}(\lambda)=0$. Then 
$$[\operatorname{soc}_{G}H^{1}(\lambda):L(\sigma)]=[\operatorname{soc}_{L} R^{1}\operatorname{ind}_{B}^{P}L_{\mathfrak f}(\lambda):L_{P}(\bar{\sigma})]$$ 
\end{theorem} 

\begin{proof} Suppose that  $R^{0}\operatorname{ind}_{B}^{P} L_{\mathfrak f}(\lambda)=0$. Consider the spectral sequence given in 
Theorem~\ref{T:spectralseq2} with $K=P$, $H=B$ and $N=L_{\mathfrak f}(\lambda)$. One has a five term exact sequence of the form 
$$0\rightarrow E_{2}^{1,0}\rightarrow E_{1} \rightarrow E_{2}^{0,1} \rightarrow E_{2}^{2,0} \rightarrow E_{2}.$$ 
The assumption implies that $E_{2}^{i,0}=0$ for $i\geq 0$. Therefore, 
$$H^{1}(\lambda)\cong \text{ind}_{P}^{G} [R^{1}\text{ind}_{B}^{P} L_{\mathfrak f}(\lambda)].$$ 
In order to compute the socle we need to consider homomorphisms of $L(\sigma)$ into $H^{1}(\lambda)$:
\begin{eqnarray*} 
\text{Hom}_{G}(L(\sigma),H^{1}(\lambda))&\cong & \text{Hom}_{G}(L(\sigma), \text{ind}_{P}^{G} [R^{1}\text{ind}_{B}^{P} L_{\mathfrak f}(\lambda)])\\
&\cong & \text{Hom}_{P}(L(\sigma), R^{1}\text{ind}_{B}^{P} L_{\mathfrak f}(\lambda))\\
&\cong & \text{Hom}_{L_{P}}(k,\text{Hom}_{U_{P}}(k,L(\sigma)^*\otimes R^{1}\text{ind}_{B}^{P}  L_{\mathfrak f}(\lambda))) \\
&\cong & \text{Hom}_{L_{P}}(k,\text{Hom}_{U_{P}}(k,L(\sigma)^*)\otimes R^{1}\text{ind}_{B}^{P}  L_{\mathfrak f}(\lambda)) \\
&\cong & \text{Hom}_{L_{P}}(k,L_{P}(\bar{\sigma})^{*}\otimes R^{1}\text{ind}_{B}^{P}  L_{\mathfrak f}(\lambda)) \\
&\cong & \text{Hom}_{L_{P}}(L_{P}(\bar{\sigma})),R^{1}\text{ind}_{B}^{P}  L_{\mathfrak f}(\lambda)).
\end{eqnarray*}
The statement of the theorem follows from this chain of isomorphisms. 
\end{proof}

\subsection{} The following result uses Theorem~\ref{T:soclecompare} and  provides a criterion for the irreducibility of $\operatorname{soc}_{G} H^{1}(\lambda)$. 

\begin{corollary}\label{C:simplesoc} Let $G$ be a supergroup arising from a simple classical Lie superalgebra ${\mathfrak g}$, $B$ be a BBW parabolic subgroup and $\lambda\in X$. Suppose there exists a parabolic subgroup scheme $P$ in $G$ with 
$B\subseteq P \subseteq G$. Assume that 
\begin{itemize} 
\item[(a)] $R^{0}\operatorname{ind}_{B}^{P} L_{\mathfrak f}(\lambda)=0$,
\item[(b)] $R^{1}\operatorname{ind}_{B}^{P} L_{\mathfrak f}(\lambda)$ has simple $L$-socle.  
\end{itemize} 
Then $\operatorname{soc}_{G} H^{1}(\lambda)$ is simple. 
\end{corollary} 

\subsection{Applications} Let $G=Q(2)$ and $\sigma = (\sigma_1, \sigma_2)$ be a weight in $(X_{\0})_{+}$.  In \cite[Section 7]{Pen}, Penkov computed the characters of all the irreducible $Q(2)$-modules $L(\sigma)$ of highest weight $\sigma$.  In particular, suppose $\sigma$ is not an integer multiple of $\rho$.  Then
$$\ch L(\sigma) =
\begin{cases}
e^0 & \sigma= 0 \\
2 \ch L_{\0}(\sigma) + 2 \ch L_{\0}(\sigma-\alpha) & \sigma_1-\sigma_2 \neq 1 \\
2 \ch L_{\0}(\sigma) & \sigma_1 - \sigma_2 = 1 \\
\end{cases},$$
where $L_{\0}(\sigma)$ is the irreducible $G_{\bar{0}}$-module of highest weight $\sigma$, with
$$\ch L_{\0}(\sigma) = e^\sigma + e^{\sigma-\alpha} + \cdots + e^{w\sigma+\alpha} + e^{w\sigma}.$$
If $\sigma$ is an integer multiple of $\rho$, then
$$\ch L(\sigma) = 2 \ch L_{\0}(\sigma).$$

On the other hand, from \cite[Lemma 4.4]{Bru}  the character of the induced modules $\text{ind}_B^G L_{\mathfrak{f}}(\sigma)$ for nonzero dominant $\sigma$ is 
$$\ch \text{ind}_B^G L_{\mathfrak{f}}(\sigma) = 2(e^\sigma + 2 e^{\sigma-\alpha} + \cdots + 2 e^{w \sigma + \alpha} + e^{w\sigma}).$$
Therefore, when $\sigma$ is not an integer multiple of $\rho$ with $\sigma_1 \neq \sigma_2+1$, $\text{ind}_B^G L_{\mathfrak{f}}(\sigma)$ is an irreducible module isomorphic to $L(\sigma)$.  
Otherwise, $\text{ind}_B^G L_{\mathfrak{f}}(\sigma)$ has length 2, with composition factors $L(\sigma)$ and $L(\sigma-\alpha)$.  Moreover, by Serre duality \cite[Theorem 5.1]{Bru}, 
$H^1(\sigma) \cong H^0(-\sigma)^*$.  If $w \sigma \in (X_{\0})_{+}$, then $H^1(\sigma)$ is either be irreducible or has socle isomorphic to $L(w\sigma-\alpha)$.  In summary, if 
$H^{1}(\sigma)\neq 0$ then $H^1(\sigma)$ will have simple $G$-socle.

Let $G=Q(n)$. Consider a minimal parabolic subgroup $P_{\alpha}$ containing $B$ with $\text{Lie }P_{\alpha}$ the Lie superalgebra generated by 
${\mathfrak b}=\text{Lie }B$ with the root spaces $x_{\alpha}$ where $\alpha\in \Delta_{\0}$ where $\Delta_{\0}=\{\epsilon_{i}-\epsilon_{i+1}:\ i=1,2,\dots,n\}$. Then 
$P_{\alpha}=L_{\alpha}\ltimes U_{\alpha}$ where $L_{\alpha}$ is a supergroup scheme of type $Q(2)$. 

Now assume that $\lambda\in X$ with $\langle \lambda,\alpha^{\vee} \rangle <0$. Then by Theorem~\ref{T:spectralseq1}, 
$$R^{0}\text{ind}_{B}^{P_{\alpha}} L_{\mathfrak f}(\lambda)|_{L_{\alpha}}\cong 
R^{0}\text{ind}_{B\cap L_{\alpha}}^{L_{\alpha}} L_{\mathfrak f}(\lambda)|_{L_{\alpha}}=0.$$ 
One can now invoke Theorem~\ref{T:soclecompare}
\begin{equation}\label{E:soclecompare}
[\operatorname{soc}_{G}H^{1}(\lambda):L(\sigma)]=[\operatorname{soc}_{L_{\alpha}} R^{1}\operatorname{ind}_{B}^{P_{\alpha}}L_{\mathfrak f}(\lambda):L_{P_{\alpha}}(\bar{\sigma})].
\end{equation}
The analysis for $Q(2)$ shows that $R^{1}\operatorname{ind}_{B}^{P_{\alpha}}L_{\mathfrak f}(\lambda)$ has a simple $L_{\alpha}$-socle. We can now state the following theorem. 

\begin{theorem} Let $G=Q(n)$ and $\lambda\in X$ where $\langle \lambda,\alpha^{\vee} \rangle <0$ for some $\alpha\in \Delta_{\0}$. 
Then $H^{1}(\lambda)$ has a simple $G$-socle. 
\end{theorem} 

\subsection{} When ${\mathfrak g}$ is a simple classical Lie superalgebra of type other than $Q$. Then a similar type of analysis can be done for minimal parabolic 
subgroups $P_{\alpha}=L_{\alpha}\ltimes U_{\alpha}$ where $L_{\alpha}$ is of type $GL(2|2)$. This motivates the following open problem. 

\begin{problem} Determine when $\text{soc}_{G} H^{1}(\lambda)$ for $G=GL(2|2)$. 
\end{problem} 

The solution to the aformentioned problem in conjunction with Corollary~\ref{C:simplesoc} would provide necessary insights into solving the more 
general problem. 

\begin{problem} Compute $\text{soc}_{G} H^{1}(\lambda)$ for all $\lambda\in X$. 
\end{problem} 

The sheaf cohomology groups $H^{n}(\lambda)$ for $n\geq 0$ are central objects for the cohomology and representation theory of $G$. 
As demonstrated in this paper, these sheaf cohomology groups unify the theory of Lie superalgebra representations. Their vanishing behavior is tied in with the 
combinatorics of the Weyl group for $G_{\0}$ acting on odd roots. Furthermore, concrete calculations are highly dependent on the use of the detecting subalgebra ${\mathfrak f}$ along with the finite reflection group $W_{\1}$. This produces a unique mixture of the odd and even theories. Further investigations along these
lines should yield solutions to the open questions about these $G$-modules and provide new insights into the representation theory for classical simple Lie superalgebras.

\end{document}